\documentclass[11pt]{amsart}
\usepackage{amsfonts,amssymb,amscd,amsmath,enumerate,verbatim,calc,latexsym,pstcol,pst-plot,pst-3d,}

\def\NZQ{\mathbb}               
\def\NN{{\NZQ N}}

\def\frk{\mathfrak}               

\def\Phi{{\frk N}}
\def\opn#1#2{\def#1{\operatorname{#2}}} 
\opn\chara{char} \opn\length{\ell} \opn\pd{pd} \opn\rk{rk}
\opn\projdim{proj\,dim} \opn\injdim{inj\,dim} \opn\rank{rank}
\opn\depth{depth} \opn\grade{grade} \opn\height{height}
\opn\embdim{emb\,dim} \opn\codim{codim}

\opn\Tr{Tr} \opn\bigrank{big\,rank}
\opn\superheight{superheight}\opn\lcm{lcm}
\opn\trdeg{tr\,deg}
\opn\reg{reg} \opn\lreg{lreg} \opn\ini{in} \opn\lpd{lpd}
\opn\size{size}\opn{\mult}{mult}
\opn\div{div} \opn\Div{Div} \opn\cl{cl} \opn\Cl{Cl}
\opn\Spec{Spec} \opn\Supp{Supp} \opn\supp{supp} \opn\Sing{Sing}
\opn\Ass{Ass} \opn\Min{Min}
\opn\Ann{Ann} \opn\Rad{Rad} \opn\Soc{Soc}
\opn\Syz{Syz} \opn\Im{Im} \opn\Ker{Ker} \opn\Coker{Coker}
\opn\Am{Am} \opn\Hom{Hom} \opn\Tor{Tor} \opn\Ext{Ext}
\opn\End{End} \opn\Aut{Aut} \opn\id{id} \opn\ini{in}

\opn\nat{nat}
\opn\pff{pf}
\opn\Pf{Pf} \opn\GL{GL} \opn\SL{SL} \opn\mod{mod} \opn\ord{ord}
\opn\Gin{Gin}
\opn\Hilb{Hilb}\opn\adeg{adeg}\opn\std{std}\opn\ip{infpt}
\opn\Pol{Pol}
\opn\sat{sat}
\opn\Var{Var}
\opn\Gen{Gen}

\opn\aff{aff} \opn\con{conv} \opn\relint{relint} \opn\st{st}
\opn\lk{lk} \opn\cn{cn} \opn\core{core} \opn\vol{vol}
\opn\link{link} \opn\star{star}
\opn\gr{gr}

\def\pot#1#2{#1[\kern-0.28ex[#2]\kern-0.28ex]}

\opn\dirlim{\underrightarrow{\lim}}
\opn\inivlim{\underleftarrow{\lim}}
\let\union=\cup
\let\sect=\cap

\let\iso=\cong

\let\to=\rightarrow

\def\Implies{\ifmmode\Longrightarrow \else
        \unskip${}\Longrightarrow{}$\ignorespaces\fi}
\def\implies{\ifmmode\Rightarrow \else
        \unskip${}\Rightarrow{}$\ignorespaces\fi}
\def\iff{\ifmmode\Longleftrightarrow \else
        \unskip${}\Longleftrightarrow{}$\ignorespaces\fi}

\let\:=\colon
\newtheorem{Theorem}{Theorem}[section]
\newtheorem{Lemma}[Theorem]{Lemma}
\newtheorem{Corollary}[Theorem]{Corollary}
\newtheorem{Proposition}[Theorem]{Proposition}

\let\epsilon\varepsilon
\let\phi=\varphi
\let\kappa=\varkappa
\def\qed{\ifhmode\textqed\fi
      \ifmmode\ifinner\quad\qedsymbol\else\dispqed\fi\fi}
\def\textqed{\unskip\nobreak\penalty50
       \hskip2em\hbox{}\nobreak\hfil\qedsymbol
       \parfillskip=0pt \finalhyphendemerits=0}
\def\dispqed{\rlap{\qquad\qedsymbol}}

\opn\dis{dis}
\def\pnt{{\raise0.5mm\hbox{\large\bf.}}}

\opn\Lex{Lex}

\newcommand{\inD}[1][\relax]{\def\argone{#1}\def\temprelax{\relax}
  \ifx\argone\temprelax\right.\else\,\middle|#1\right.{}\fi}

\newif\ifbinary
\binarytrue

\begin{document}

\title{Indispensable Hibi relations and Gr\"obner bases}

\author{Ayesha Asloob Qureshi}

\address{Ayesha Asloob Qureshi, Abdus Salam School of Mathematical Sciences,
GC University, Lahore.
68-B, New Muslim Town, Lahore 54600, Pakistan} \email{ayesqi@gmail.com}

\begin{abstract}
In this paper we consider Hibi rings and Rees rings attached to a poset. We classify the ideal lattices of posets whose Hibi relations are indispensable and the ideal lattices of posets whose Hibi relations form a quadratic Gr\"obner basis with respect to the rank lexicographic order. Similar classifications are obtained for Rees rings of Hibi ideals.
\end{abstract}

\maketitle

\section*{Introduction}

The main purpose of this paper is to classify those distributive lattices with the property that the Hibi relations are indispensable and those with the property that Hibi relations form a Gr\"obner basis with respect to the rank lexicographic order. To be precise let $L$ be a finite lattice. Attached to this lattice one defines the so-called Hibi ideal as follows: we fix a field $K$ and consider the polynomial ring $T=K[\{z_a\:\; a\in L\}]$ over $K$ whose variables are indexed by the elements of $L$. Then
\[
I_L=(z_az_b-z_{a\wedge  b}z_{a\vee b}:\; a,b\in L).
\]
is called the {\em Hibi relation ideal} of $L$. Relations of the form $z_az_b-z_{a\wedge  b}z_{a\vee b}$  are called {\em Hibi relations}.

The $K$-algebra
\[
{\mathcal R}_K[L]=T/I_L
\]
is called the {\em  Hibi ring} of $L$ (over $K$).

We order variables in $T=K[\{z_a\:\; a\in L\}]$ such that $z_a < z_b$ if $\rank a < \rank b$ and call any monomial order induced by this ordering the {\em rank order}.

In \cite{H1}, Hibi proved the following fundamental fact which says that the $K$-algebra ${\mathcal R}_K[L]$ is a domain (hence a toric ring) if and only if $L$ is distributive. In fact Hibi showed that for distributive lattice Hibi relations form the reduced Gr\"obner basis with respect to the reverse lexicographic order. Even though Hibi relations generate $I_L$, they may not be indispensable in the sense of Hibi and Ohsugi \cite{HO}. In other words, in general there may exist a minimal set of generators of $I_L$ consisting of relations other than Hibi relations. The simplest example of such a lattice is the Boolean lattice $B_3$ which consists of all the subsets of a three element set.

In Theorem~\ref{hot} we give the classification of finite distributive lattices with the property that for $I_L$ the Hibi relations are indispensable. To describe the result, recall that according to Birkhoff's theorem every finite distributive lattice is isomorphic to the ideal lattice of a finite poset. This poset is uniquely determined by $L$. In fact, it is the subposet $P$ of $L$ consisting of join-irreducible elements of $L$. Among other equivalent conditions for the property that Hibi relations are indispensable, it is shown in Theorem~\ref{hot} that all poset ideals of $P$ are generated by at most 2 elements. Another equivalent condition says that $L$ is a conditionally URC lattice. Modifying the definition of uniquely complemented lattices given by Stanley in \cite{S}, we call a lattice $L$ conditionally uniquely relatively complemented (conditionally URC), if each interval $[a,b]$ in $L$ has unique complements provided they exist. Recall that $c,d \in [a,b]$ are called complements of each other with respect to $[a,b]$ if $c\vee d= b$ and $c \wedge d= a$. In Theorem~\ref{URC}, we observe that a conditionally URC lattice is always distributive. We show in Proposition~\ref{URC} that a URC lattice is isomorphic to a sublattice of $\NN^2$ of the form $[m]_0 \times [n]_0$, where $[k]_0= \{0, 1, \ldots, n\}$.

Motivated by the paper \cite{AHH} of Aramova, Herzog and Hibi where it is shown in \cite[Theorem 2.5]{AHH} that the Hibi ring of a finite simple planar distributive lattice has a quadratic Gr\"obner basis if and only if $L$ is a chain ladder, we classify in Theorem~\ref{defense} all distributive lattices $L$ having the property that the reduced Gr\"obner basis of $I_L$ consists of Hibi relations. One of the equivalent condition states that $L$ is a chain ladder without critical corner.

Let $P=\{p_1, \ldots, p_n\}$ be a finite poset and $L$ be its ideal lattice. In the last section of the paper we study the Gr\"obner basis of the defining ideal $J_L$ of the Rees ring of the Hibi ideal $H_L$. The Hibi ideal $H_L$ is defined to be the monomial ideal generated by the monomials $u_a=\prod_{p_i\in a}x_i\prod_{p_i\not \in a}y_i$ in the polynomial ring $K[x_1,\ldots, x_n, y_1, \ldots, y_n]$. In \cite{HH1}, the Gr\"obner basis of $J_L$ is described with respect to the rank reverse lexicographic order. The main result of Section 4 is Theorem~\ref{Rees} where it is shown that a distributive lattice $L$ is a URC lattice if and only if the reduced Gr\"obner basis with respect to natural lexicographic order consists of Hibi relations and special linear relations. This result is used in Corollary~\ref{meet} to study for meet-distributive meet-semilattice $L$, the reduced Gr\"obner basis of $J_L$ with respect to a lexicographic order.

\section{Hibi rings with indispensable Hibi relations}

In this section we want to classify all distributive lattices $L$ with the property that the Hibi relations $z_az_b-z_{a\wedge  b}z_{a\vee b}$ are indispensable, which means  that the Hibi relations appear in each minimal binomial set of generators of $I_L$. Before discussing this problem we recall some fundamental facts about Hibi rings.

Let $L$ be a finite distributive lattice. According to Birkhof's theorem, the distributive lattice $L$ is isomorphic to the ideal lattice of the subposet $P$ of $L$ consisting of all join irreducible elements of $L$. Thus we may always view $L$ as the ideal lattice $\mathcal{I}(P)$ of a poset $P$. Say, $P=\{p_1,\ldots,p_r\}$, and let $S= K[x_1,\ldots, x_r,y_1,\ldots,y_r]$ be  the polynomial ring in $2r$ indeterminate. For each $a\in L$ we define the monomial
\begin{eqnarray}\label{monomial}
u_a=\prod_{p_i\in a}x_i\prod_{p_i\not \in a}y_i,
\end{eqnarray}
and consider the $K$-algebra homomorphism
\[
\varphi\:\; T\to S,\quad z_a\mapsto u_a.
\]
Then one shows that $\Ker(\varphi)=I_L$, where $I_L=(z_az_b-z_{a\wedge  b}z_{a\vee b}:\; a,b\in L)$. Hence ${\mathcal R}_K[L]\iso K[\{u_a\:\; a\in L\}]$, which implies ${\mathcal R}_K[L]$ is a domain. In fact Hibi showed that the Hibi relations form a reduced Gr\"obner basis of $\Ker(\varphi)$ with respect to reverse rank lexicographic order, see \cite{H1} and \cite[Theorem 10.1.3]{HH2}.

Note that a lattice is distributive if and only if it does not contain one of the following sublattices shown in Figure \ref{Fig1}.

\begin{figure}[h]
\begin{center}
\psset{unit=0.5cm}
\begin{pspicture}(-9,-1)(4,3)
\psline(2,3)(0.5,1.5)
\psline(0.5,1.5)(0.5,0)
\psline(0.5,0)(2,-1)
\psline(2,-1)(3.5,0.75)
\psline(3.5,0.75)(2,3)
\rput(2,3){$\bullet$}
\put(1.9,3.3){$a$}
\rput(0.5,1.5){$\bullet$}
\put(-0.1,1.5){$b$}
\rput(0.5,0){$\bullet$}
\put(-0.1,0){$c$}
\rput(2,-1){$\bullet$}
\put(1.9,-1.6){$e$}
\rput(3.5,0.75){$\bullet$}
\put(3.8,0.75){$d$}
\psline(-9,3)(-11,1)
\psline(-9,3)(-9.5,1)
\psline(-9,3)(-7,1)
\psline(-11,1)(-9,-1)
\psline(-9.5,1)(-9,-1)
\psline(-7,1)(-9,-1)
\rput(-9,3){$\bullet$}
\put(-9.1,3.3){$a$}
\rput(-11,1){$\bullet$}
\put(-11.5,1){$b$}
\rput(-9.5,1){$\bullet$}
\put(-10.1,1){$c$}
\rput(-7,1){$\bullet$}
\put(-6.7,1){$d$}
\rput(-9,-1){$\bullet$}
\put(-9.1,-1.6){$e$}
\end{pspicture}
\end{center}
\caption{}\label{Fig1}
\end{figure}

Assume now that $L$ is not a distributive lattice. Then it contains at least one of the sublattices as shown in Figure \ref{Fig1}. Say, it contains the sublattice on the left, then $z_bz_c-z_az_e, z_bz_d-z_az_e \in I_L$, which implies $z_b(z_c-z_d) \in I_L$, but neither $z_b$ or $z_c-z_d$ belongs to $I_L$. Hence $I_L$ is not a prime ideal in this case. Similarly it can be seen that $I_L$ is not prime if $L$ contains the sublattice on the right.

Distributive lattices are characterized as follows.

\begin{Proposition}
Let $L$ be a lattice. Then the following conditions are equivalent:
\begin{enumerate}
\item[{\em (a)}] L is a distributive lattice.
\item[{\em (b)}] Hibi relations form a Gr\"obner basis with respect to the rank reverse lexicographic order.
\end{enumerate}
\end{Proposition}

\begin{proof}
It suffice to proof (b) $\Rightarrow$ (a): Suppose $L$ is not a distributive lattice. Then it contains at least one of the sublattices as shown in Figure \ref{Fig1}. Say, it contains sublattice on the right, then $z_bz_d-z_az_e$, $z_cz_d-z_az_e \in I_L$. Therefore $f=z_az_ez_b-z_az_ez_c \in I_L$. On the other hand $\ini_<(f)=z_az_ez_b$ is not divided by any initial term of a Hibi relation in $I_L$.
\end{proof}

\medskip
Now we come back to the main problem of this section concerning the indispensability of Hibi relations. For example, consider the Boolean lattice $B_3$, see Figure~\ref{Fig2}, which is the ideal lattice of the poset consisting of an anti-chain with three elements.

\begin{figure}[h]
\begin{center}
\psset{unit=0.8cm}
\begin{pspicture}(1,-1.5)(5,3)
\psline(3,3)(1.5,1.5)
\psline(3,3)(3,1.5)
\psline(3,3)(4.5,1.5)
\psline(1.5,1.5)(1.5,0)
\psline(1.5,1.5)(3,0)
\psline(3,1.5)(1.5,0)
\psline(3,1.5)(4.5,0)
\psline(4.5,1.5)(4.5,0)
\psline(4.5,1.5)(3,0)
\psline(1.5,0)(3,-1.5)
\psline(3,0)(3,-1.5)
\psline(4.5,0)(3,-1.5)
\rput(3,3){$\bullet$}
\put(3,3.2){$h$}
\rput(1.5,1.5){$\bullet$}
\put(1.2,1.5){$e$}
\rput(3,1.5){$\bullet$}
\put(2.6,1.5){$f$}
\rput(1.5,0){$\bullet$}
\put(1.1,-0.1){$b$}
\rput(4.5,1.5){$\bullet$}
\put(4.7,1.5){$g$}
\rput(3,0){$\bullet$}
\put(2.5,-0.1){$c$}
\rput(4.5,0){$\bullet$}
\put(4.7,-0.1){$d$}
\rput(3,-1.5){$\bullet$}
\put(2.8,-1.9){$a$}
\end{pspicture}
\end{center}
\caption{}\label{Fig2}
\end{figure}

The two Hibi relations $z_e z_d-z_a z_h$, $z_g z _b - z_a z_h $ can be replaced by the relations $z_e z_d - z_g z_b$, $z_g z_b - z_a z_h$ where the first of them is not a Hibi relation. Hence in this example, the Hibi relations are not indispensable.

We need some preparations to prove the main theorem of this section.

\begin{Lemma}
\label{first}
Let $L$ be distributive lattice and $f=z_a z_b - z_c z_d$ be a non-zero element in $I_L$. Then $a\wedge b = c\wedge d$ and $a\vee b= c\vee d$. In particular, if $c$ and $d$ are comparable, then $f$ is a Hibi relation.
\end{Lemma}

\begin{proof}
For a monomial $u \in S=K[x_1,\ldots , x_n,y_1,\ldots, y_n]$ we set
\[
\supp_{x} (u)=\{x_i\:\; \text{$x_i$ divides $u$}\} \quad \text{and} \quad  \supp_{y} (u)=\{y_i\:\; \text{$y_i$ divides $u$}\}.
\]

Since $f\in \Ker(\phi)$, we have
\[
\supp_{x}(u_a u_b) = \supp_{x}(u_c u_d)\quad \text{and}\quad \supp_{y}(u_a u_b) = \supp_{y}(u_c u_d),
\]
where for $e \in L$, $u_e$ denotes the monomial defined as in (\ref{monomial}).

This implies that $a\wedge b = c\wedge d$ and $a\vee b= c\vee d$.
\end{proof}

In order to formulate the main result of this section we have to introduce some notation and concepts.
Let $L$ be a lattice and  $[a,b]$ be an interval of $L$ and $c,d \in [a,b]$. Then $d$ is called a {\em complement} of $c$ with respect to $[a,b]$ if $d \vee c=b$ and $d \wedge c =a$. The set $\{c,d\}$ is called a {\em complementary set} of $[a,b]$, if $\{c,d\} \neq \{a,b\}$. An interval is {\em complemented} if it admits a complementary set.

\begin{Lemma}
\label{Complement}
Let $L$ be a distributive lattice, $[a,b]$ an interval of $L$ and $c \in [a,b]$. Suppose c has a complement with respect to $[a,b]$, then this complement is uniquely determined.
\end{Lemma}

\begin{proof}
The proof follows from the fact a distributive lattice does not contain a sublattice as shown in Figure~\ref{Fig1}.
\end{proof}

We call a lattice $L$ {\em uniquely relatively complemented} or a { \em URC-lattice}  if for every interval $ [a,b]$ of $L$ either $[a , b ]$ is a chain or there exists a unique complementary set $\{c,d\}$ of $[a,b]$.  The lattice $L$ is said to be a {\em conditionally URC-lattice}, if for each interval $[a,b]$ of $L$, a complementary set of $[a,b]$ is unique provided it exists.

The following figures show an example of a URC-lattice and a conditionally URC-lattice.

\begin{figure}[h]
\begin{center}
\psset{unit=0.8cm}
\begin{pspicture}(1,-2)(5,4)
\rput(-4,0){\psline(3,1)(2,2)
\psline(2,2)(3,3)
\psline(3,3)(4,2)
\psline(4,2)(3,1)
\psline(3,1)(2,0)
\psline(2,0)(3,-1)
\psline(3,-1)(4,0)
\psline(4,0)(3,1)
\psline(4,2)(5,1)
\psline(5,1)(4,0)
\psline(2,2)(1,1)
\psline(1,1)(2,0)
\rput(3,-2){URC Lattice}
}
\rput(-4,0){\rput(3,1){$\bullet$}
\rput(2,2){$\bullet$}
\rput(3,3){$\bullet$}
\rput(4,2){$\bullet$}
\rput(2,0){$\bullet$}
\rput(3,-1){$\bullet$}
\rput(4,0){$\bullet$}
\rput(5,1){$\bullet$}
\rput(1,1){$\bullet$}
}
\rput(4,0){\psline(3,1)(2,2)
\psline(2,2)(3,3)
\psline(3,3)(4,2)
\psline(4,2)(3,1)
\psline(3,1)(2,0)
\psline(2,0)(3,-1)
\psline(3,-1)(4,0)
\psline(4,0)(3,1)
\rput(3,1){$\bullet$}
\rput(2,2){$\bullet$}
\rput(3,3){$\bullet$}
\rput(4,2){$\bullet$}
\rput(2,0){$\bullet$}
\rput(3,-1){$\bullet$}
\rput(4,0){$\bullet$}
\rput(3,-2){Conditionally URC lattice}
}
\end{pspicture}
\end{center}
\caption{}\label{Fig3}
\end{figure}

\begin{Theorem}
\label{URC}
A URC lattice is distributive.
\end{Theorem}

\begin{proof}
The proof follows from the fact that a URC lattice does not contain any sublattice shown in Figure~\ref{Fig1}.
\end{proof}

In the case that $L$ is a distributive lattice, the conditionally URC property can be characterized as follows.

\begin{Lemma}
\label{neighbors}
Let $L$ be a distributive lattice $L$. Then the following conditions are equivalent:
\begin{enumerate}
\item[{\em (a)}] For all $y \in L$, $y$ has at most two lower neighbors.
\item[{\em (b)}] For all $x \in L$, $x$ has at most two upper neighbors.
\item[{\em (c)}] $L$ is conditionally URC.
\end{enumerate}
\end{Lemma}

\begin{proof}
(a)$\Rightarrow$(b):
Suppose $x \in L$ has three distinct upper neighbors,  say,  $l,m,n$. Since $L$ is distributive, it follows that $l \vee n \vee m$ has at least three distinct lower neighbors, namely, $l\vee m$, $l\vee n$ and $m\vee n$. This leads to contradiction to our assumption.

(b)$\Rightarrow$(a) is proved similarly.

(b)$\Rightarrow$(c): Suppose $L$ is not conditionally URC. Then there exists an interval $[a,b]$ of $L$ such that it has two distinct complementary sets  $\{c_1,c_2\}$ and $\{d_1,d_2\}$. It follows from Lemma~\ref{Complement}, that $\{c_1,c_2\} \cap \{d_1,d_2\} = \emptyset$.

Assume that one of the  $c_i$ is comparable with one of the $d_j$, say, $c_1<d_2$. Then $c_1 \wedge d_1 = a$, because  $a \leq c_1\wedge d_1 \leq  d_2 \wedge d_1 = a$. Then $c_1 \vee d_1 <b$. Let $b_1$ and $b_2$ be the two lower neighbors of $b$,  and $a_1$ and $a_2$ be the two upper neighbors of $a$. We may assume that $d_1< b_1$ and $d_2 <b_2$. We have $c_1 \vee d_1 \leq b_1 < b$ which implies $c_1 < b_1$. Since we assume that $c_1 < d_2$, we also get $c_1 < b_2$. On the other hand, $c_2 < d_1$ or $c_2 < d_2$, which gives $c_1 \vee c_2 < b$, a contradiction.

So, $c_1,c_2,d_1,d_2$ are pairwise incomparable. We may assume that $c_1,d_1 < b_1$ and $c_2,d_2<b_2$. Clearly, $c_1 \vee d_2 = b $. It follows from Lemma~\ref{Complement} that $c_1 \wedge d_2 > a$. We can assume that $c_1 \wedge d_2 \geq a_1 >a$ which gives $c_1 , d_2 \geq a_1$ and $c_2,d_1 \geq a_2$. This implies that $c_1\wedge d_1=a$,  since  $c_1\not\geq a_2$ and $d_1\not\geq a_1$. Distributivity of $L$ gives $d_1 = (c_1 \vee d_2) \wedge d_1 = (c_1 \wedge d_1)\vee(d_2 \wedge d_1) = a$, a contradiction.

(c)$\Rightarrow$(a): Suppose there exists $x \in L$ such that $x$ has at least three lower neighbors, say, $a,b,c$.  Since $L$ is distributive it follows that
\[
a \wedge b \neq b \wedge c \neq c \wedge a.
\]
The sets $\{a \wedge b ,c\}$,  $\{b \wedge c,a\}$ are distinct complementary sets of interval $[a \wedge b\wedge c,x]$, a contradiction.
\end{proof}

For an integer $k \geq 0$, we set $[k]_0 = \{0,1,\ldots,k\}$. Now we can state the main result of this section.

\begin{Theorem}
\label{hot}
Let $P$ be a finite poset and $L$ its ideal lattice. The following conditions are equivalent:

\begin{enumerate}
\item[{\em (a)}] For $I_L$ the Hibi relation are indispensable.
\item[{\em (b)}] $L$ is conditionally URC.
\item[{\em (c)}] In the poset $P$, all poset ideals are generated by at most $2$ elements.
\item[{\em (d)}] The poset $P$ can be covered by two disjoint chains, i.e,  we have chains $C$ and $D$ in $P$ such that $ V(P)= V(C) \cup V(D)$ and $V(C)\cap V(D) = \emptyset$.
\item[{\em (e)}] $L$ can be embedded as a full sublattice in $[m]_0\times [n]_{0}$, where $m =|C|$ and $n=|D|$.
\end{enumerate}
\end{Theorem}

\begin{proof}
(a)$\Rightarrow $(b): Suppose that $L$ is not  conditionally URC. Then there exist an interval $[a,b]$ of $L$ such that it has two distinct complementary sets  $\{x,y\}$ and $\{r,s\}$. For these two sets, we have two  Hibi relations $h_1=z_x z_y - z_a z_b$ and $h_2 = z_r z_s - z_a z_b$ in $I_L$ which implies that $h_3 = z_x z_y - z_r z_s \in I_L$. The relation $h_3$ is not a Hibi relation and $h_1=h_2+h_3$. It shows that $h_1$ is dispensable.

(b)$\Rightarrow$(a): Let $L$ be a conditionally URC lattice and$H$ be the set of all Hibi relations in $I_L$. Take $f\in H$ where $f = z_cz_d-z_az_b$ and $\{c,d\}$ is a complementary set of $[a,b]$. Suppose $f$ is dispensable. Then it can be written as a $K$-linear combination of some other degree 2 binomials $g_1, \ldots , g_n$ in $I_L$ with $g_i\neq f$ for all $i$. It follows that $z_a z_b \in\supp g_i$ for some $i \in [n]$, say, $g_i=z_r z_s - z_a z_b$. From Lemma~\ref{first}, we know that $g_i$ must be a Hibi relation, i.e, $r \wedge s =a$ and $r \vee s =b$. Since $L$ is conditionally URC, we must have $\{c,d\}=\{r,s\}$. It gives $f=g_i$, a contradiction.

(b) $\Rightarrow$ (c): Suppose there exists a poset ideal $(p,q,r)$  of $P$ which is minimally generated by three elements. Clearly, $p,q$ and $r$ are incomparable in $P$. Let $b =(p,q,r)$. Then $b$ has three lower neighbors in $L$, namely $b / \{r\}$, $b / \{p\}$ and $b / \{q\}$, which contradicts Lemma~\ref{neighbors}.

(c)$\Rightarrow$(d): We choose a chain of ideals $\emptyset =a_0\subset a_1\subset a_2\subset \ldots \subset a_s=P$ with $\sharp(a_i \setminus a_{i-1}) = 1 $, for all $i$. Each $a_i$ may be viewed as subposet of $P$ which also satisfies condition (c).  Thus by induction on the cardinality of the poset we may assume that $a_{s-1}$ can be covered by two disjoint chains, say $C_0$ and $D_0$ with maximal elements $q$ and $r$ respectively. Take $p \in P$ such that $a_{s}=a_{s-1}\union\{p\}$.

Suppose that $p$ is comparable with either $q$ or $r$, say comparable with $q$. Then we let $C=C_0 \cup \{p\}$ and $D=D_0$. Otherwise we may assume that there exist a lower neighbor of $p$ in $D_0$ different from $r$. Let $D_0=\{ d_1, d_2, \ldots, d_k\}$ with $d_1 < d_2 < \ldots < d_k$. Suppose that the lower neighbor of $p$ in $D_0$ is $d_i$ with $i<k$. It follows that $d_{i+1}$ is comparable with $q$, because otherwise $(p,q,d_{i+1})$ is a 3-generated ideal, contradicting our assumption (c). In both cases, namely $q < d_{i+1}$ and $q > d_{i+1}$, we define $C= C_0 \cup \{d_{i+1}, \ldots , d_k\}$ and $D=\{d_1, \ldots, d_i, p\}$. Note that, if $q> d_{i+1}$, then $C_0 \cup \{d_{i+1}, \ldots, d_k\}$ is a chain. Otherwise, for any $c_i$ incomparable with some $d_{i+l}$ and $i+l < k$, we have $c_i$ incomparable with $p$, because $c_i < p$ gives $c_i \leq d_i < d_{i+l}$. Then the ideal $(c_i, d_{i+l}, p)$ is 3-generated ideal, a contradiction.

(d)$\Rightarrow$(e): Let $C$ and $D$ be given by $c_1< \ldots < c_n$ and $d_1<\ldots<d_m$ respectively.  We define the embedding $\varphi\: L\to \NN^2$ by
\[
\varphi(a)= \left\{ \begin{array}{ll}
        (i,j), & \text{if  $a \cap C=(c_i)$ and $a \cap D=(d_j)$}, \\
        (0,j), & \text{if  $a \cap C=\emptyset$ and $a \cap D=(d_j)$},\\
        (i,0), & \text{if  $a \cap C=(c_i)$ and $a \cap D=\emptyset$},\\
        (0,0), & \text{if  $a \cap C=\emptyset$ and $a \cap D=\emptyset$}.
        \end{array} \right.
\]
Observe first that $\varphi$ is injective. Indeed, if $\varphi(a)=\varphi(b)$, then $a\sect C=b\sect C$ and $a\sect D=b\sect D$. Since $P=C\union D$,  we then have
\[
a=a\sect P=a\sect (C\union D)=(a\sect C)\union (a\sect D)= (b\sect C)\union (b\sect D)=b\sect (C\union D)=b\sect P=b.
\]
Next we show that $\varphi(a\wedge b)=\varphi(a)\wedge \varphi(b)$. Let $\varphi(a)=(i,j)$ and $\varphi(b)=(k,l)$. Then,
\[
((a \wedge b) \cap C , (a \wedge b) \cap D) = ((a \cap C) \cap (b \cap C),(a \cap D) \cap (b \cap D))=(c_{\min\{i,k\}},d_{\min\{j,l\}}).
 \]
Therefore, $\varphi(a\wedge b)=(\min\{i,k\},\min\{j,l\})=\varphi(a)\wedge \varphi(b)$. For the join the argument is similar.

Now it remains to be shown that the embedding yields a full sublattice of $[m]_0\times [n]_{0}$, where $n =|C|$ and $m=|D|$. In other words we have to show that $\varphi(L)$ contains a chain of length $n+m$. For this consider the chain of ideals in $P$ which we introduced in the proof (c) $\Rightarrow$ (d). By construction, this chain has length $|P|=n+m$. Therefore $\varphi (a_0) < \varphi(a_1) < \cdots < \varphi(a_{n+m})$ is the desired chain in $\varphi(L)$.

(e)$\Rightarrow$(b): Let $(i,j) \in L$. Since $L$ is full sublattice of $[m]_0\times [n]_{0}$, it follows that each upper neighbor of $(i,j)$ is of the form $(i+1,j)$ or $(i,j+1)$. So the assertion follows from Lemma \ref{neighbors}.
\end{proof}
An interesting special case of the previous theorem is described in the next result.
\begin{Proposition}
\label{URC}
Let $P$ be a finite poset and $L$ be it ideal lattice. Then following conditions are equivalent.
\begin{enumerate}
\item[{\em (a)}] $L$ is a URC lattice.
\item[{\em (b)}] Either $P$ is a chain or it consists of two disjoint chains $C$ and $D$ such that all elements of $C$ are incomparable with all elements of $D$.
\item[{\em (c)}] There exist non-negative integers $m$ and $n$ such that $L \iso [m]_0 \times [n]_0$.
\end{enumerate}
\end{Proposition}

\begin{proof}
(a)$\Rightarrow$(b): From Theorem~\ref{hot}, we know that there exist two disjoint chains $C$ and $D$ which cover $P$. Assume that $P$ does not satisfy $(b)$. Then $P$ contains two incomparable elements, say $p_1 \in C$ and $p_2 \in D$. Moreover, there exist $c \in C$ and $d \in D$ such that they are comparable. We may assume that $c_i > d_j$.

Suppose that $P$ has only one minimal element, say $q$. The interval $[\emptyset, (p_1, p_2)]$ of $L$ is not a chain because it contains two incomparable elements $(p_1)$ and $(p_2)$. Moreover, this interval does not have a complementary set because the only upper neighbor of $\emptyset$ in $L$ is$(p)$, a contradiction.

Now suppose that $P$ has two minimal elements, say $q_1 \in C$ and $q_2 \in D$. It follows that $c > q_1, q_2$. Let $c'$ be the minimal element in $C$ with this property. Then $c'$ has two incomparable lower neighbors $r_1$ and $r_2$ in $P$. Therefore it follows that the interval $[(r_1)\cap(r_2), (c')]$ of $L$ is not a chain and does not have a complementary set, because $(r_1,r_2)$ is the only lower neighbor of $(c')$ in $[(r_1)\cap(r_2), (c')]$, again a contradiction.

(b) $\Rightarrow$ (c): If $P$ is a chain then $L \iso [m]_0 \times [0]_0$. Otherwise, $P$ is the disjoint union of two chains $C:c_1 < c_2 < \ldots < c_m$ and $D: d_1 < d_2 < \ldots < d_n$, where none of the $c_i$ is comparable with any of the $d_j$. As in the proof of (d) $\Rightarrow$ (c) of Theorem~\ref{hot}, we have the embedding $\varphi:L \rightarrow [m]_0 \times [n]_0$. To show that $\varphi$ is an isomorphism it is enough to show that $|L|=(m+1)(n+1)$. To see this we observer that if $\alpha \in L$ then $\alpha= \emptyset$ or $\alpha= (c_i)$ or $\alpha= (d_j)$ or $\alpha= (c_i,d_j)$. It is obvious that ideals $\emptyset$, $(c_i)$, $(d_j)$ are pairwise distinct, and that these ideals are also different from the 2-generated ideal $(c_i,d_j)$. Suppose now that $(c_i,d_j)=(c_k,d_l)$. Since the elements of $C$ are all incomparable with elements of $D$, it follows that $c_i \leq c_k$ and $d_j \leq d_l$. Similarly one has $c_k \leq c_i$ and $d_l \leq d_j$. Altogether we conclude that $|L|=(m+1)(n+1)$.

(c)$\Rightarrow$(a): Let $L \iso [m]_0 \times [n]_0$ for some non-negative integers $m$ and $n$. To show that $L$ is indeed a URC lattice, it is enough to show that every interval in $L$ which is not a chain has a complementary set. Let $[(i,j),(k,l)]$ be an interval in $L$ with $i<k$ and $j<l$. There exist two incomparable elements $a,b \in L$, namely $a=(k,j)$ and $b=(i,l)$ with $a \wedge b=(i,j)$ and $a \vee b=(k,l)$.
\end{proof}

\section{Gr\"obner bases of Hibi rings with respect to rank lexicographic orders}

In this section we want to classify all distributive lattices with the property that with respect to the rank lexicographic order the Hibi ideal of the lattice has a reduced Gr\"obner basis consisting of Hibi relations.

In order to formulate our main result we introduce some terminology. Let $L$ be a full sublattice of $[m]_0\times [n]_{0}$. Let $(i,j)$ be an element in $L$ such that $(i-1,j),(i+1,j),(i,j+1),(i,j-1)$ also belong to L. We call it an {\em upper corner} if $(i-1,j+1) \notin L$ and $(i+1,j-1) \in L$, a {\em lower corner} if $(i-1,j+1) \in L$ and $(i+1,j-1) \notin L$ and {\em critical corner} if $(i-1,j+1) \notin L$ and $(i+1,j-1) \notin L$. A lattice $L$ is called a {\em chain ladder}, (see \cite{AH}), if all upper corners and lower corners appear in a chain and that, for any two corners $(i,j) \neq (i',j')$ of $D$, one has $i \neq i'$ and $j \neq j'$.

\begin{Theorem}
\label{defense}
Let $L$ be a distributive lattice. The following conditions are equivalent:
\begin{enumerate}
\item[{\em (a)}] The  reduced  Gr\"obner basis of $I_L$ with respect to a rank lexicographic order consists of all Hibi relations in $I_L$.
\item[{\em (b)}] The Hibi relations are indispensable, and $I_L$ has a reduced quadratic Gr\"obner basis with respect to a rank lexicographic order.
\item[{\em (c)}] $L$ is conditionally URC, and for all $a<b<c$ in $L$ such that $[a,b]$ and $[b,c]$ have complementary sets, it follows that either $[g,c]$ or $[h,c]$ is complemented, where $\{g,h\}$ is the complementary set of $[a,b]$.
\item[{\em (d)}]  $L$ is isomorphic to a chain ladder without critical corners.
\end{enumerate}
\end{Theorem}

\begin{proof}
(a)$\Rightarrow$ (b): We have a quadratic Gr\"obner basis since Hibi relations are quadratic.
Suppose that $f$ is a quadratic binomial relation with $ \ini(f)=z_az_b$. It follows from (a) that $a$ and $b$ are comparable. Therefore Lemma~\ref{first} implies that $f=z_a z_b-z_c z_d$,where $[c,d]$ is complementary pair of $[a,b]$, as desired.

(b)$\Rightarrow$ (a): Let $f$ be a binomial in reduced Gr\"obner basias of $I_L$. By our assumption $f$ is a quadratic binomial. Since Hibi relations are indispensable $f$ must be a Hibi relation.

(b)$\Rightarrow$(c): From Theorem~\ref{hot}, we know that $L$ is conditionally URC and it can be identified with a full sublattice in $[m]_0\times [n]_{0}$. Let $a)$, $b$ and $c$ be the elements in $L$ such that $[a,b]$ and $[b,c]$ are complemented with complementary pairs $\{g,h\}$, and $\{d,e\}$, respectively.

Consider the S-polynomial $z_c z_g z_h - z_a z_e z_d$ of the Hibi relations $z_a z_b - z_g z_h$ and $z_b z_c - z_e z_d$. The monomial $z_cz_gz_h$ is the leading term of the S-polynomial. Since by our assumption the Gr\"obner basis of $I_L$ consists of Hibi relations, it follows that there exits a Hibi relation with initial term $z_cz_g$ or $z_cz_h$. This implies that the interval $[g,c]$ or $[h,c]$ is complemented.

(c)$\Rightarrow$(d): Since $L$ is a conditionally URC, we may identify it with a full sublattice in $[m]_0 \times [n]_0$. Suppose $L$ has a critical corner $b=(i,j)$. By definition of critical corner $(i-1,j),(i+1,j),(i,j+1),(i,j-1)\in L$. Therefore, since $L$ is a lattice, $a=(i-1,j-1)$ and $c=(i+1,j+1)$ belong to $L$. Let $[a,b] = [(i-1,j-1),(i,j)]$ and $[b,c]=[(i,j),(i+1,j+1)]$, and $d=(i,j+1)$ and $e=(i+1,j)$.
Since $(i-1,j+1) \notin L$ and $(i+1,j-1) \notin L$, it follows that $[a,d]$ and $[a,e]$ are not complemented, a contradiction.

It remains to show that $L$ is a chain ladder. First, suppose that $L$ has two incomparable corners, $x=(i,j)$ and $y=(k,l)$. Then we may assume $i<k$, $j >l$. Since $L$ is a lattice it contains also the elements $w =x \wedge y=(i,l)$ and $z=x \vee y=(k,j) $ and since $L$ is a full sublattice of $[m]_0 \times [n]_0$, it contain all elements $\{(r,s)\:\;i \leq r \leq k, l \leq s \leq j \}$. This implies $x$ is an upper corner and $y$ is a lower corner. By definition of corners, it follows that $d=(i-1,j)$, $e=(i,j+1)$, $g=(k+1,l)$ and $f=(k-1,l)$ belong to $L$. Hence $a=d\wedge f=(i-1,l-1)$ and $c= e \vee g=(k+i,j+i)$ belong to $L$. Now we have $a<b<c$. The interval $[a,b]$, $[b,c]$ are complemented. Therefore, either the interval $[d,c]$ and $[f,c]$ must be complemented by our assumption (c), contradicting the fact that $x$ and $y$ are upper and lower corners respectively.

Now suppose $L$ has two corners $a=(i,j)$ and $b=(k,l)$ such that either $i=k$ or $j=l$. Let $j=l$. We can assume that $i<k$. It gives $a<b$. By the definition of corners, the elements $(i-1,j)$, $i,j-1$, $(i+1,j)$, $(i,j+1)$ and $(k+1,j)$, $(k,j+1)$, $(k-1,j)$, $(k,j-1)$ belong to $L$. Since $L$ is a full sublattice of $[m]_0 \times [n]_0$, it follows that $[(i,j-1), (k,j-1)] \subset L$. In particular $(i+1,j-1) \in L$. This shows $a$ is an upper corner. Similarly one shows that $b$ is a lower corner. Since $L$ is a lattice $c=(i-1,j-1)$ and $d=(k+1,j+1)$ also belong to $L$. We have $c<b<d$ and also the intervals $[c,b]$ and $[b,d]$ are complemented. From (c), we know that either $[(k,j-1), d]$ or $[(i-1,j+1), d]$ must be complemented, in other words, either $(i-1,j+1)$ or $(k+1,j-1)$ must belong to $L$. This contradicts our supposition. A similar argument holds if we assume $i=k$.

(d)$\Rightarrow$(b): It is shown \cite[Theorem 2.5]{AHH} that $I_L$ has a quadratic Gr\"obner basis under the additional assumption that $L$ is simple. In the same way it is shown that $L$ has quadratic Gr\"obner basis even if it is not simple, provided it satisfies (d). Since $L$ is a conditionally URC, it follows from Lemma~\ref{hot}, that Hibi relations are indispensable.
\end{proof}

\section{Rees rings of Hibi ideals}

Let $L$ be the ideal lattice of the poset $P=\{p_1, \ldots, p_n\}$, and $S=K[\{x_{p_i},y_{p_i}\}_{{p_i}\in P}]$ be the polynomial ring in $2n$ variables over a field $K$ with $\deg x_{p_i} = \deg y_{p_i} =1$. Recall that to each element $a \in L$, we associate a squarefree monomial $u_a=\prod_{p_i\in a}x_i\prod_{p_i\not \in a}y_i$ and the Hibi ideal $H_L$ is defined to be the ideal of $S$ generated by such monomials, i.e. $H_L =( u_a | a \in L)$, see \cite{HH1}.

Let $\mathcal{R}(H_L)$ denote the Rees algebra of $H_L$ and $J_L$ be the defining ideal of $\mathcal{R}(H_L)$. In other words, $\mathcal{R}(H_L)$ is the affine semigroup ring given by
\[
\mathcal{R}(H_L)=S[\{u_{a}t\}_{a \in L}]= K[\{x_{p_i},y_{p_i}\}_{{p_i} \in P},\{u_{a}t\}_{a \in L}]  \subset K[\{ x_{p_i}, y_{p_i}\}_{{p_i} \in P},t],
\]
and $J_L$ is the kernel of the surjective ring homomorphism $\phi : R \rightarrow \mathcal{R}(H_L)$ where
\[
R=S[\{z_a\}_{a \in L}]=K[\{x_{p_i},y_{p_i}\}_{{p_i} \in P},\{z_{a}\}_{a \in L}]
\]
is a polynomial ring over $K$ and $\phi$ is defined by setting
\begin{eqnarray}\label{phi}
\phi(x_{p_i}) =x_{p_i} , \, \phi(y_{p_i})=y_{p_i} ,\,  \phi (z_a)=u_{a}t.
\end{eqnarray}

In this section we are interested in the Gr\"obner basis of $J_L$ with respect to a suitable lexicographical orders. We define a term order on $R=K[\{x_{p_i},y_{p_i}\}_{{p_i} \in P},\{z_{a}\}_{a\in L}]$ and for the sake of convenience we write $x_i, y_i$ instead of $x_{p_i} ,y_{p_i}$. The term order on $R$, denoted by $<^1_{lex}$, is defined to be the product order of the lexicographic order on $S$ induced by $x_1>\cdots>x_n>y_1>\cdots> y_n$ and a rank lexicographical order on $T$. In particular $x_i >^1_{lex} y_j >^1_{lex} z_a$ for all $i$, $j$ and $a$.

Let $a_1$ and $a_2$ be two poset ideals of $P$ such that $a_2=a_1 \cup \{p_i\}$. To each such pair of poset ideals, we associate a binomial $x_i z_{a_1} - y_i z_{a_2}$, and call it a {\em special linear relation} in $R$.

Now we state the main theorem of this section.
\begin{Theorem} \label{Rees}
Let $L$ be a distributive lattice. Then following conditions are equivalent.
\begin{enumerate}
\item[{\em (a)}] $L$ is a URC lattice.
\item[{\em (b)}] The reduced Gr\"obner basis of $J_L$ with respect to $<^1_{lex}$ consists of Hibi relations and special linear relations.
\end{enumerate}
\end{Theorem}

\begin{proof}
(a) $\Rightarrow$ (b): From \cite[Theorem 1.1]{HH1} and its proof, we know that $J_L$ is minimally generated by Hibi relations and special linear relations. Let $M$ be the set of these relations. To show that $M$ is a reduced Gr\"obner basis of $J_L$ with respect to $<^1_{lex}$, we must show that all S-pairs $S(f_i,f_j) , 1\leq i,j\leq n$ reduce to $0$. Take $f_i,f_j \in M$ and consider the non-trivial case when $\gcd(\ini_{<}(f_i),\ini_{<}(f_j)) \neq 1$. For any binomial, we always write the leading term as the first term.

If $f_i$ and $f_j$ are both Hibi relation then $S(f_i,f_j)$ reduces to 0 because of Theorem~\ref{defense}. Next we consider the case that $f_i$ is a Hibi relation and $f_j$ is a special linear relation. Say,
\[
f_i= z_d z_a - z_b z_c \quad \text{with} \quad d>a, \quad \text{and} \quad f_j =x_p z_a- y_p z_e \quad \text{or} \quad f_j= x_{p}z_d - y_{p}z_e.
 \]
Let us first assume that $f_j=x_p z_a- y_p z_e$. Then it follows from the relation $f_j$ that $a$ is a lower neighbor of $e$. From Proposition~\ref{URC}, we know that $L \iso [m]_0 \times [n]_0$. Let $b=(i,j)$ and $c=(k,l)$ with $i<k$ and $j>l$. Then $a=(i,l)$ and $d=(k,j)$. Since $a$ is a lower neighbor of $e$, we have $e=(i,l+1)$ or $e=(i+1,l)$. Assume $e=(i,l+1)$. Take $f = (k,l+1)$. Then $c$ is a lower neighbor of $f$, see Figure~\ref{Fig4}.

\begin{figure}[h]
\begin{center}
\psset{unit=1.2cm}
\begin{pspicture}(1,1)(5,3)
\psline(3,1)(1.75,2.25)
\psline(1.75,2.25)(2.75,3.25)
\psline(2.75,3.25)(4,2)
\psline(4,2)(3,1)
\psline(3.5,2.5)(2.5,1.5)
\rput(3,1){$\bullet$}
\put(2.95,0.7){$a$}
\rput(1.75,2.25){$\bullet$}
\put(1.5,2.25){$b$}
\rput(2.75,3.25){$\bullet$}
\put(2.7,3.35){$d$}
\rput(4,2){$\bullet$}
\put(4.1,2){$c$}
\rput(2.5,1.5){$\bullet$}
\put(2.2,1.3){$e$}
\rput(3.5,2.5){$\bullet$}
\put(3.6,2.6){$f$}
\put(2.8,1.4){$p$}
\put(3.5,2){$p$}
\end{pspicture}
\end{center}
\caption{}\label{Fig4}
\end{figure}

If $b=e$, then we also have $d=f$ and we obtain
\[
S(f_i,f_j)=x_p z_b z_c - y_p z_d z_e = z_b (x_p z_c - y_p z_d).
\]
Therefore $S(f_i,f_j)$ reduces to 0.

Now, if $b>e$, then we first observe $\{b,f\}$ is the complementary set in $[e,d]$. Therefore, in this case
\[
S(f_i,f_j)= x_p z_b z_c - y_p z_d z_e = z_b ( x_p z_c - y_p z_f) - y_p (z_d z_e - z_b z_f ).
\]
It shows that $S(f_i,f_j)$ again reduces to zero.

Next assume that $f_j=x_p z_d - y_p z_e$. It follows from the relation $f_j$ that $d$ is lower neighbor of $e$. Let $b=(i,j)$ and $c=(k,l)$ with $i<k$ and $j>l$. Then $a=(i,l)$ and $d=(k,j)$ and either $e=(k,j+1)$ or $e=(k+1,j)$. We can assume that $e=(k+1,j)$. Since the interval $[a,e]$ has the complementary set $\{b,g\}$, the interval $[c,e]$ has the complementary set $\{d,g\}$ where $g=(k+1,l)$, see Figure~\ref{Fig5}.

\begin{figure}[h]
\begin{center}
\psset{unit=1.2cm}
\begin{pspicture}(0,0.5)(6,3)
\psline(3,1)(2,2)
\psline(2,2)(3,3)
\psline(3,3)(4,2)
\psline(4,2)(3,1)
\psline(3,1)(2.5,0.5)
\psline(2,2)(1.5,1.5)
\psline(1.5,1.5)(2.5,0.5)
\rput(3,1){$\bullet$}
\put(3.08,.75){$c$}
\rput(2,2){$\bullet$}
\put(1.75,2.08){$d$}
\rput(3,3){$\bullet$}
\put(3,3.15){$e$}
\rput(4,2){$\bullet$}
\put(4.1,1.9){$g$}
\rput(2.5,0.5){$\bullet$}
\put(2.4,0.2){$a$}
\rput(1.5,1.5){$\bullet$}
\put(1.25,1.5){$b$}
\put(2.55,0.85){$p$}
\put(1.8,1.6){$p$}
\end{pspicture}
\end{center}
\caption{}\label{Fig5}
\end{figure}
Therefore, we have
\[
S(f_i,f_j)= x_p z_b z_c - y_p z_e z_a= z_b ( x_p z_c - y_p z_g)- y_p (z_e z_a - z_g z_b)
\]
Again, $S(f_i,f_j)$ reduces to 0.

Now, we consider the case when both $f_i$ and $f_j$ are special linear relations. Say,
\[
f_i= x_p z_a - y_p z_b \quad \text{and} \quad f_j= x_q z_a - y_q z_c \quad \text{or} \quad f_j= x_p z_d - y_p z_e
\]
First assume that $f_j= x_q z_a - y_q z_c$. Let $d=b \vee c$, see Figure~\ref{Fig6}.
\begin{figure}[h]
\begin{center}
\psset{unit=1.2cm}
\begin{pspicture}(0,1)(6,3)
\psline(3,1)(2,2)
\psline(2,2)(3,3)
\psline(3,3)(4,2)
\psline(4,2)(3,1)
\rput(3,1){$\bullet$}
\put(2.9,0.65){$a$}
\rput(2,2){$\bullet$}
\put(1.75,2){$c$}
\rput(3,3){$\bullet$}
\put(2.9,3.15){$d$}
\rput(4,2){$\bullet$}
\put(4.1,1.95){$b$}
\put(2.32,2.6){$p$}
\put(3.54,2.6){$q$}
\put(2.28,1.35){$q$}
\put(3.58,1.42){$p$}
\end{pspicture}
\end{center}
\caption{}\label{Fig6}
\end{figure}

Then $S(f_i,f_j)= x_p y_q z_c - x_q y_p z_b = y_q(x_p z_c - y_p z_d) - y_p (x_q z_b - y_q z_d)$. Therefore, $S(f_i,f_j)$ reduces to 0.

Now, take $f_j= x_p z_d - y_p z_e$. We can assume that $b>e$. Take $a=(i,j)$, $b=(i+1,j)$, $d=(i,l)$ and $e=(i+1,l)$ where $j>l$, see the Figure~\ref{Fig7}.

\begin{figure}[h]
\begin{center}
\psset{unit=1.2cm}
\begin{pspicture}(0,1)(6,4)
\psline(3,1)(1.5,2.5)
\psline(1.5,2.5)(2.5,3.5)
\psline(2.5,3.5)(4,2)
\psline(4,2)(3,1)
\rput(3,1){$\bullet$}
\put(2.9,0.63){$d$}
\rput(1.5,2.5){$\bullet$}
\put(1.2,2.45){$a$}
\rput(2.5,3.5){$\bullet$}
\put(2.4,3.6){$b$}
\rput(4,2){$\bullet$}
\put(4.1,1.9){$e$}
\put(2.08,2.85){$p$}
\put(3.27,1.59){$p$}
\end{pspicture}
\end{center}
\caption{}\label{Fig7}
\end{figure}

Then $\{a,e\}$ is the complementary set in $[d,b]$ and we have
\[
S(f_i,f_j)= y_p z_a z_e - y_p z_b z_d = -y_p ( z_b z_d - z_a z_e )
\]
Hence $S(f_i,f_j)$ reduces to 0. This complete the proof.

(b)$\Rightarrow$(a): Since $<^1_{lex}$ is an elimination order for the variables $x_i$ and $y_j$, it follows that the Gr\"obner basis of $J_L \cap T$ with respect to the rank lexicographic order consists elements of the Gr\"obner basis of $J_L$ with respect to $<^1_{lex}$ which belong to $T$. By assumption (b) these relations are exactly the Hibi relations in $J_L$. Thus, the Gr\"obner basis with respect to the rank lexicographical order of the Hibi relation ideal of the Hibi ring $\mathcal{R}_K(L)$ (which is $J_L \cap T$), consists of Hibi relations. Therefore, from Theorem~\ref{defense}, we know that $L$ is a chain ladder without critical corners. Let $m$ and $n$ be the non-negative integers such that $L$ has an embedding in $[m]_0 \times [n]_0$ and $(m,n)$ is the maximal element in $L$. Then it is enough to show that $L$ has no upper or lower corners because then $L \iso [m]_0 \times[n]_0$.

Suppose $L$ has upper or lower corners. Let $C$ be the maximal chain of upper and lower corners in $L$ with maximal element $a$. Let $a=(i,j)$ and $b=(m,n)$. Then $[a,b]$ is complemented in $L$. Take $\{c,d\}$ be the complementary set of $[a,b]$. We can assume that $a$ is an upper corner in $L$, i.e., $(i-1,j+1)\notin L$. Then, the elements $e=(i-1,j)$, $g=(i,j-1)$, $f=(i-1,j-1)$, and $c=(i,j+1)$ belong to $L$. Consider the $S$-polynomial of the binomials $f_i=z_a z_f - z_e z_g$ and $f_j = z_a x_p - z_c y_p$ in $J_L$, where $c= a \cup \{p\}$. Then $S(f_i,f_j)= x_p z_e z_g - z_f z_c y_p$ reduces to $0$ if and only if $x_p z_e - y_p z_h \in J_L$, where $h=(i-1, j+1)$. This implies $(i-1,j+1) \in L$, a contradiction to our assumption.
\end{proof}

In the following we extend the previous result to meet-distributive meet-semilattices. Recall that a poset $L$ is called a {\em meet-semilattice} if every pair of elements of $L$ has a meet in $L$. A finite meet-semilattice $L$ is called {\em meet-distributive} if each interval $[x, y]$ of
$L$ such that $x$ is the meet of the lower neighbors of $y$  in this interval is Boolean. Let $P$ be the set of join irreducible elements in $L$. For any $l \in L$, we call the cardinality of $\{p \in P| p \leq l\}$ the {\em degree} of $l$, and the maximum of the lengths of chains descending from $l$ the {\em rank} of $l$. $L$ is called {\em graded} if all maximal chains have the same length. In \cite{Ede}, the following characterization of meet-distributive meet-semilattices is given.

\begin{Lemma}\label{meet-dist}
For a finite lattice L the following conditions are equivalent:
\begin{enumerate}

\item[{\em (a)}] $L$ is meet-distributive.
\item[{\em (b)}] $L$ is graded and $\deg l = \rank l$, for all $l \in L$.
\item[{\em (c)}] Each element in $L$ is a unique minimal join of join-irreducible elements.
\end{enumerate}
\end{Lemma}

The above lemma shows that a distributive lattice is also a meet-distributive meet-semilattice.

Let $L$ be a meet-distributive meet-semilattice and $P$ be the poset consisting of all the join-irreducible elements in $L$. We denote by $\hat{L}$ the ideal lattice of $P$ and call it {\em associated distributive lattice} of $L$. We have a canonical embedding of $L$ in $\hat{L}$ given by $l \mapsto \{p \in P| p \leq l\}$ for all $l \in L$.

\begin{Proposition}
Let $L$ be a meet-distributive meet-semilattice and $\hat{L}$ be its associated distributive lattice. Then $L$ is a poset ideal of $\hat{L}$.
\end{Proposition}

\begin{proof}
Take $s \in \hat{L}$ and $r \in L$ such that $s \leq r$. From Lemma~\ref{meet-dist}, we have $\rank_L r = \deg_L r$. Also, we have $\deg_L r=\deg_{\hat{L}} r=\rank_{\hat{L}} r$, which gives and $\rank_L r = \rank_{\hat{L}} r$. It shows any maximal chain descending from $r$ in $\hat{L}$ also survives in $L$. Hence, we obtain $s \in L$.
\end{proof}

 We denote by $H_L$ the ideal of $S$ generated by  monomials $u_a$ with $a\in L$ as described in (\ref{monomial}). Let $\mathcal{R}(H_L)$ denote the Rees algebra of $H_L$ and $J_L$ be the defining ideal of $\mathcal{R}(H_L)$. We have $H_L \subset H_{\hat{L}}$ and $\mathcal{R}(H_L) \subset \mathcal{R}(H_{\hat{L}})$.

\begin{Corollary} \label{meet}
Let $L$ be a meet-distributive meet-semilattice. Suppose that the associated distributive lattice $\hat{L}$ of $L$ is a URC lattice. Then the following conditions are equivalent:

\begin{enumerate}
\item[{\em (a)}] $L=\hat{L}$.
\item[{\em (b)}] The reduced Gr\"obner basis of $J_L$ with respect to $<^{1}_{lex}$ consists of Hibi relations and special linear relations.
\end{enumerate}

\end{Corollary}

\begin{proof}
(a) $\Rightarrow$ (b) follows from Theorem~\ref{Rees}.

(b) $\Rightarrow$ (a): Assume $L \subsetneq \hat{L} $. Since $L$ is a poset ideal of $\hat{L}$ and $\hat{L} \iso [m]_0 \times [n]_0 $, there exist two incomparable elements $a, b \in L$ such that they cover $c=a \wedge b$ and $d=a \vee b \notin L$. Let $a=c\cup \{p\}$ and $b=c\cup \{q\}$ with $p,q \in P$. Then $f_i=x_pz_c-y_pz_a$ and $f_j=x_qz_c-y_qz_b$ are special linear relations in $J_L$, and
\[
S(f_i,f_j)=x_py_qz_b-x_qy_pz_a
\]
with the initial monomial $x_py_qz_b$ if $x_p > x_q$, as we may assume. Our assumption $(b)$ implies that the initial monomial of some Hibi relation or special linear relation must divide $x_py_qz_b$. It follows that the only special linear relation whose initial term divides $x_py_qz_b$ is $x_p z_b - y_p z_d$. Since $d \notin L$, we arrive at a contradiction.
\end{proof}

\end{document}